\documentclass[11pt]{amsart}
\usepackage{amssymb}

\usepackage{palatino}
\input amssym.def

\usepackage{amsmath, amsfonts}
\usepackage{amssymb}
\usepackage{amscd}
\usepackage[mathscr]{eucal}
\usepackage{palatino}
\newfont{\cyrr}{wncyr10}

\newcommand{\p}{{\mathfrak P}}
\newcommand{\N}{{\mathbb N}}

\newcommand{\Q}{{\mathbb Q}}
\newcommand{\R}{{\mathbb R}}
\newcommand{\C}{{\mathbb C}}
\newcommand{\B}{{\mathfrak B}}

\newcommand{\thmref}[1]{Theorem~\ref{#1}}
\newcommand{\propref}[1]{Proposition~\ref{#1}}

\newcommand{\tab}{\qquad }

\newtheorem{thm}{Theorem}

\newtheorem{lem}[thm]{Lemma}
\newtheorem{cor}[thm]{Corollary}
\newtheorem{prop}[thm]{Proposition}
\newtheorem{rmk}{Remark}[section] 
\newtheorem{defn}{Definition}

\begin{document}

\title[Simultaneous Sign Change and Non-vanishing]{The first 
simultaneous sign change and
non-vanishing of Hecke eigenvalues of newforms}

\author{Sanoli Gun, Balesh Kumar and Biplab Paul}

\address[Sanoli Gun, Balesh Kumar and Biplab Paul]  
{Institute of Mathematical Sciences, 
Homi Bhabha National Institute, 
C.I.T Campus, Taramani, 
Chennai  600 113, India.}
\email{sanoli@imsc.res.in}
\email{baleshk@imsc.res.in}
\email{biplabpaul@imsc.res.in}



\begin{abstract} 
Let $f$ and $g$ be two distinct newforms
which are normalized Hecke eigenforms
of weights $k_1, k_2 \ge 2$ 
and levels $N_1, N_2 \ge 1$ respectively.  
Also let $a_f(n)$ and $a_g(n)$
be the $n$-th Fourier-coefficients of $f$ and $g$ 
respectively. In this article, we investigate the first 
sign change of the sequence 
$\{a_f(p^{\alpha})a_g(p^{\alpha}) \}_{p^{\alpha} \in \N, \alpha \le 2}$,
where $p$ is a prime number.
We further study the non-vanishing of the sequence
$\{a_f(n)a_g(n) \}_{n \in \N}$ and derive 
bounds for first non-vanishing term in this sequence.
We also show, using ideas of Kowalski-Robert-Wu 
and Murty-Murty,
that there exists a set of primes  $S$ of natural density one
such that for any prime $p \in S$, the
sequence $\{a_f(p^n)a_g(p^m) \}_{n,m \in \N}$
has no zero elements. This improves
a recent work of Kumari and Ram Murty. Finally, using
$\B$-free numbers, we investigate simultaneous non-vanishing
of coefficients of $m$-th symmetric power $L$-functions 
of non-CM forms in short intervals.
\end{abstract}

\maketitle

\section{Introduction}
For positive integers $k \ge 2, N \ge 1$, let
$S_k(N)$ be the space of cusp forms of weight $k$ for the 
congruence subgroup $\Gamma_0(N)$ and 
$S_k^{new}(N)$ be the subspace of $S_k(N)$ 
consisting of newforms. We investigate arithmetic
properties of Fourier-coefficients of $f \in S_k^{new}(N)$
which are normalized Hecke eigenforms. This question
has been studied extensively by several mathematicians.
In recent works, Kowalski, Lau, Soundararajan and Wu \cite{KLSW} 
and later Matom\"aki \cite{M} showed that any 
$f \in S_k^{new}(N)$ which is a normalized Hecke eigenform 
is uniquely determined by the signs of their
Hecke eigenvalues at primes. In this article, we 
investigate simultaneous sign change and non-vanishing 
of Hecke eigenvalues of such forms. 
More precisely, for $z \in \mathcal{H}:=\{z\in \C : \text{Im}(z)>0\}, 
~q:= e^{2\pi i z}$, let
\begin{equation}\label{forms}
f(z) =\sum_{n=1}^\infty a_f(n)q^n \in S_{k_1}^{new}(N_1)
\phantom{m}\text{ and }\phantom{m}
g(z) =\sum_{n=1}^\infty a_g(n)q^n \in S_{k_2}^{new}(N_2)
\end{equation}
be two newforms which are normalized Hecke eigenforms. 
Here we study first sign change 
and non-vanishing of the sequence $\{a_f(n)a_g(n)\}_{n \in \N}$.

The question of simultaneous sign change for arbitrary cusp
forms was first studied
by Kohnen and Sengupta \cite{KS} under certain conditions
which were later removed by the first author, Kohnen and 
Rath \cite{GKR}. In the later paper, the authors 
prove infinitely many sign change of the sequence 
$\{a_f(n)a_g(n)\}_{n\in\N}$. Here we prove the
following theorem.

\begin{thm}\label{th1}
Let $N_1, N_2$ be square-free, $N:=\text{lcm}[N_1,N_2]$
and $f \in S_{k_1}^{new}(N_1), ~g \in S_{k_2}^{new}(N_2)$ be 
two distinct normalized Hecke eigenforms with Fourier expansions as
in \eqref{forms}. Then there exists a prime power
$p^{\alpha}$ with $\alpha \le 2$ and
$$
p^{\alpha} \ll_\epsilon 
\text{max}
\left\{exp{(c \log^2 (\sqrt{\mathfrak{q}(f)} + \sqrt{\mathfrak{q}(g)} ))}, 
\left[ N^2
\left(1+\frac{k_2-k_1}{2}\right)
\left(\frac{k_1+k_2}{2}\right) \right]^{1+\epsilon} \right\}
$$
such that $a_f(p^{\alpha})a_g(p^{\alpha})<0$.
Here $c >0$ is an absolute constant and
$\mathfrak{q}(f), ~\mathfrak{q}(g)$
are analytic conductors
of Rankin-Selberg $L$-functions of $f$ and $g$ respectively.
\end{thm}

We use Rankin-Selberg method
and an idea of Iwaniec, Kohnen and Sengupta \cite{IKS} 
to prove \thmref{th1}.  This theorem
can be thought of as a variant 
of Strum's result about distinguishing two
newforms by their Fourier-coefficients.  
This result can be compared with the
results of Lau-Liu-Wu \cite{LLW}, Kohnen~\cite{K}, 
Kowalski-Michel-Vanderkam \cite{KMV},
Ram Murty~\cite{RM2} and
Sengupta~\cite{JS}.

Next we investigate sign changes  of
the sequence $\{a_f(n)a_g(n^2)\}_{n \in \N}$ in short 
intervals. This question of sign change 
for the sequence $\{a_f(n)a_g(n)\}_{n \in \N}$
in short intervals was considered by Kumari and 
Ram Murty (see \cite[Theorem 1.6]{KM}). 
Here we prove the following.

\begin{thm}\label{th_2}
Let 
$$
f(z) =\sum_{n=1}^\infty a_f(n)q^n \in S_{k_1}^{new}(N_1)
\phantom{m}\text{ and }\phantom{m}
g(z) =\sum_{n=1}^\infty a_g(n)q^n \in S_{k_2}^{new}(N_2)
$$ 
be two distinct normalized Hecke eigenforms. 
For any sufficiently large $x$ and any $\delta>\frac{17}{18}$, 
the sequence $\{a_f(n)a_g(n^2)\}_{n \in \N}$
has at least one sign change in $(x, x + x^{\delta}]$. 
In particular, the number of sign 
changes for $n\leq x$ is $\gg x^{1-\delta}$.
\end{thm}

Sign changes of Hecke eigenvalues implies 
non-vanishing of Hecke eigenvalues. The question of
non-vanishing of Hecke eigenvalues has been
studied by several mathematicians. 
One of the fundamental open problem in this area is a question of 
Lehmer which predicts that $\tau(n)\neq 0$ for all $n\in\N$, 
where $\tau(n)$ is the Ramanujan's $\tau$-function defined as 
follows;
$$
\Delta(z)
:=\sum_{n=1}^\infty \tau(n)q^n
:=q\prod_{n=1}^\infty (1-q^n)^{24}.
$$
It is well known that $\Delta(z) \in S_{12}(1)$ is 
the unique normalized Hecke eigenform. We now
investigate non-vanishing of the sequence
$\{a_f(p^m)a_g(p^m)\}_{m \in \N}$ and our first 
theorem in this direction is the following.
\begin{thm}\label{th_3}
Let 
$$
f(z) =\sum_{n=1}^\infty a_f(n)q^n \in S_{k_1}^{new}(N_1)
\phantom{m}\text{ and }\phantom{m}
g(z) =\sum_{n=1}^\infty a_g(n)q^n \in S_{k_2}^{new}(N_2)
$$ 
be two distinct normalized Hecke eigenforms. 
Then for all primes $p$ with $(p, N_1N_2)=1$, the set 
\begin{equation}\label{densityofset}
\{m\in \N ~|~  a_f(p^m)a_g(p^m)\ne 0 \}
\end{equation}
has positive density. 
\end{thm}

The first author along with Kohnen and Rath 
(see Theorem 3 of \cite{GKR})
showed that   
for infinitely many primes $p$, the sequence
$\mathcal{A}_p :=\{a_f(p^m)a_g(p^m)\}_{ m \in \N}$
has infinitely many sign changes and hence
in particular, $\mathcal{A}_p$ has infinitely many 
non-zero elements. 
\thmref{th_3} shows that for all primes $p$ with
$(p, N_1N_2)=1$, the non-zero elements of the
sequence $\mathcal{A}_p$ has positive density
and hence does not follow from Theorem~3 of~\cite{GKR}.
Our next theorem strengthens Theorem 1.2 of 
Kumari and Ram Murty \cite{KM}. 

\begin{thm}\label{th_4}
Let 
$$
f(z) =\sum_{n=1}^\infty a_f(n)q^n  \in S_{k_1}^{new}(N_1)
\phantom{m}\text{ and }\phantom{m}
g(z) =\sum_{n=1}^\infty a_g(n)q^n \in S_{k_2}^{new}(N_2)
$$ 
be two distinct normalized non-CM Hecke eigenforms. 
Then there exists a set $S$ of 
primes with natural density one such that for any $p \in S$ and 
integers $m, m'  \ge 1$, we have 
$$
a_f(p^{m})a_g(p^{m'}) \neq 0.
$$ 
\end{thm}

Now we shall consider the question of the first 
simultaneous non-vanishing which is analogous 
to the question considered 
in Theorem 1. Our result here is as follows.

\begin{thm}\label{th_5}
Let 
$$
f(z) =\sum_{n=1}^\infty a_f(n)q^n \in S_{k_1}(N_1)
\phantom{m}\text{ and }\phantom{m}
g(z) =\sum_{n=1}^\infty a_g(n)q^n \in S_{k_2}(N_2)
$$ 
be two distinct normalized Hecke eigenforms.
Further assume that $N:=\text{ lcm }[N_1,N_2] > 12$.
Then there exists a positive integer    
$1< n \le (2\log N)^{4}$ with $(n, N)=1$
such that 
$$
a_{f}(n)a_{g}(n) \neq 0. 
$$
Further, when $N$ is odd, then
there exists an integer $1 < n \le 16$ with $(n, N)=1$
such that 
$$
a_{f}(n)a_{g}(n) \neq 0. 
$$
\end{thm}

Note that $a_f(1)a_g(1) =1 \ne 0$ but we are trying
to find the first natural number $n >1$ with $(n,N)=1$
for which $a_f(n)a_g(n) \ne 0$
which we call the first non-trivial simultaneous non-vanishing.
Though first simultaneous sign change (see \cite{LLW}, also \thmref{th1}
above) implies first non-trivial simultaneous non-vanishing but the bound
proved in \thmref{th_5} is much stronger
for first non-trivial simultaneous non-vanishing.

The paper is organized as follows. 
In the next section, we introduce notations and 
briefly recall some preliminaries. In sections 3 to 7, we provide
proofs of theorems mentioned in the introduction. 
Finally, in the last section, using $\B$-free numbers, we
deduce certain results about
simultaneous non-vanishing of coefficients of symmetric 
power $L$-functions of non-CM forms in short intervals.

\section{Notation and Preliminaries}

Throughout the paper, $p$ denotes a prime number
and $\mathcal{P}$ denotes the set of all primes. We say 
that a subset $S\subset \mathcal{P}$ has natural density $d(S)$ 
if the limit 
$$
\lim_{x\to \infty}
\frac{\#\{p\in\mathcal{P}: p\le x \text{ and } p\in S\}}
{\#\{p\in\mathcal{P} : p\le x\}}
$$
exists and equal to $d(S)$. For any non-negative real number 
$x$, we denote the greatest integer $n \le x$ by $[x]$. Let $A$ 
be a subset of the set of natural numbers. Then we say the 
density of the set $A$ is $d(A)$ if the limit
$$
\lim_{x\to \infty}\frac{\#\{n\le x : n\in A\}}{\#\{n\le x\}}
$$
exists and equal to the real number $d(A)$. For any $n,m\in\N$, 
we shall denote the greatest common divisor of $n$ and $m$ by 
$(n,m)$.
 
For a normalized Hecke eigenform 
$f \in S_k^{new}(N)$ 
with Fourier expansion 
$$
f(z)=\sum_{n\ge 1}a_f(n)q^n, 
$$
we write 
$$
\lambda_f(n):=\frac{a_f(n)}{n^{(k-1)/2}}.
$$ 
From the theory of Hecke operators, we know
\begin{equation}\label{eq2}
\lambda_f(1)=1
\phantom{m}\text{ and }\phantom{m}
\lambda_f(m)\lambda_f(n)
=\sum_{d|(m,n),(d,N)=1}\lambda_f\left(\frac{mn}{d^2}\right).
\end{equation}
Also by a celebrated work of Deligne, we have 
\begin{equation}\label{eq3}
|\lambda_f(n)|\le d(n) \text{ for all }(n,N)=1,
\end{equation}
where $d(n)$ denotes the number of positive divisors of $n$.

The following result of Kowalski-Robert-Wu
\cite[Lemma 2.3]{KRW} (see also 
Murty-Murty \cite[Lemma 2.5]{MM} )
plays an important role in the proof of \thmref{th_4}. 

\begin{lem}\label{lem6}
Let 
$$
f(z):=\sum_{n=1}^\infty a_f(n)q^n \in S_k^{new}(N)
$$
be a normalized non-CM Hecke eigenform.
For $\nu \ge 1$, let 
$$
P_{f,\nu}:=\{p\in\mathcal{P} ~|~ p\nmid N 
\text{ and } \lambda_f(p^\nu)= 0\}.
$$
Then for any $\nu\ge 1$, we have 
$$
\#(P_{f,\nu}\cap [1,x])\ll_{f,\delta} \frac{x}{(\log x)^{1+\delta}}
$$
for any $x\ge 2$ and $0<\delta<1/2$. Here the implied constant 
depends on $f$ and $\delta$. Let 
$$
P_f:=\cup_{\nu\in\N}P_{f,\nu}.
$$
Then for any $x\ge 2$ and $0<\delta<1/2$, 
we have 
$$
\#(P_f\cap[1,x])
\ll_{f,\delta} 
\frac{x}{(\log x)^{1+\delta}},
$$
where the implied constant 
depends only on $f$ and $\delta$.
\end{lem}

We now recall some well known properties of Rankin-Selberg 
$L$-function associated with $f\in S_{k_1}^{new}(N_1)$ 
and $g\in S_{k_2}^{new}(N_2)$ which are
normalized Hecke eigenforms. Suppose that $k_1~\le~k_2$.
One can now define the 
Rankin-Selberg $L$-function as follows
$$
R(f,g; s):=\sum_{n\ge 1}\lambda_f(n)\lambda_g(n)n^{-s},
$$
which is absolutely convergent for $\Re(s)>1$ and hence 
it defines a holomorphic function there. Let 
$M:= \text{gcd}(N_1,N_2)$ and $N:= \text{lcm}[N_1,N_2]$ be 
square-free. 
By the work of Rankin \cite{RR} (see also \cite{AO}, page 304), 
one knows that the function 
$\zeta_N(2s)R(f,g;s)$ is entire if $f\neq g$, where $\zeta_N(s)$ 
is defined by 
\begin{equation}\label{zeta}
\zeta_N(s):=\prod_{p \nmid N}\left(1-p^{-s}\right)^{-1}
\phantom{m}\text{ for  }~ \Re(s) > 1.
\end{equation}
We also have 
the completed Rankin-Selberg $L$-function 
\begin{equation}
R^*(f,g;s):=(2\pi)^{-2s}\Gamma(s+\frac{k_2-k_1}{2})
\Gamma(s+\frac{k_1+k_2}{2}-1)
\prod_{p|M}(1-c_pp^{-s})^{-1}\zeta_N(2s) R(f,g; s)
\end{equation} 
with $c_p=\pm 1$ depending on the forms $f$ and $g$.
It is well known by the works of Ogg (see \cite[Theorem 6]{AO}) and 
Winnie Li (see \cite[Theorem 2.2]{WL}) that the  
completed Rankin-Selberg $L$-function  
satisfies the functional equation 
\begin{equation}\label{eq_5}
R^*(f,g; s)=N^{1-2s}R^*(f,g; 1-s).
\end{equation}

\section{Proof of \thmref{th1}}

Throughout this section, we assume that $N_1$ and $N_2$ 
are square-free and $f \in S_{k_1}^{new}(N_1)$, 
$g \in~S_{k_2}^{new}(N_2)$  
are two distinct normalized Hecke eigenforms with
$1 < k_1 \le k_2$.  In order to prove \thmref{th1},
we need to prove the following 
Propositions. 

\begin{prop}
For square-free integers $N_1,N_2$, let $f\in S_{k_1}^{new}(N_1),
g\in S_{k_2}^{new}(N_2)$ be normalized Hecke eigenforms with 
$f\neq g$ and let $N:= \text{lcm}[N_1,N_2]$ and $M:= (N_1, N_2)$. 
Then for any $t\in \R$ and 
$\epsilon>0$, one has 
\begin{eqnarray*}
&& \zeta_N(2+2\epsilon+ 2it)R(f,g; 1+\epsilon+it)\ll_{\epsilon} 1\\
&\text{ and }& 
\zeta_N(-2\epsilon + 2it)R(f,g; -\epsilon + it)
\ll_{\epsilon} 
N^{2+4\epsilon}
\left(1+\frac{k_2-k_1}{2}\right)^{1+2\epsilon}
\left(\frac{k_1+k_2}{2}\right)^{1+2\epsilon}
|1+ it |^{2(1+2\epsilon)},
\end{eqnarray*}
where $\zeta_N(s)$ is defined in \eqref{zeta}.
\end{prop}

\begin{proof}
Since $\zeta_N(2+2\epsilon+ 2it)$ and $R(f,g; 1+\epsilon+it)$ 
are absolutely convergent for $\Re(s)>1$, we have the first inequality.
To derive the second inequality, we use functional equation.
From the functional equation \eqref{eq_5}, we have 
\begin{eqnarray}\label{eq6}
\zeta_N(2-2s) \cdot R(f,g; 1-s)
&=& (2\pi)^{2-4s}\cdot N^{2s-1}
\cdot \frac{\Gamma(s+\frac{k_2-k_1}{2})}{\Gamma(1-s+\frac{k_2-k_1}{2})}
\cdot \frac{\Gamma(s+\frac{k_1+k_2}{2}-1)}{\Gamma(-s+\frac{k_1+k_2}{2})}\\
&&
\cdot \prod_{p|M}\left(\frac{1-c_pp^{s-1}}{1-c_pp^{-s}}\right)
\cdot \zeta_N(2s) \cdot R(f,g; s).\nonumber
\end{eqnarray}
Using Stirling's formula (see page 57 of \cite{AI}), we have 
\begin{eqnarray*}
&&\left|\frac{\Gamma(1+\frac{k_2-k_1}{2}+\epsilon+it)}
{\Gamma(\frac{k_2-k_1}{2}-\epsilon+it)}\right|
\ll_\epsilon \left(1+\frac{k_2-k_1}{2}\right)^{1+ 2\epsilon}|1+it|^{1+2\epsilon}\\
&\text{ and }& 
\left|\frac{\Gamma(\frac{k_1+k_2}{2}+\epsilon+it)}
{\Gamma(\frac{k_1+k_2}{2}-1-\epsilon+it)}\right| 
\ll_\epsilon \left(\frac{k_1+k_2}{2}\right)^{1+2\epsilon}|1+it|^{1+2\epsilon}.
\end{eqnarray*}
For all $t\in \R$, we also have 
\begin{eqnarray*}
&&\left|\prod_{p|M}(1-c_pp^{-1-\epsilon-it})^{-1}\right|
=\left|\prod_{p|M}\sum_{m\ge 0}(c_pp^{-1-\epsilon-it})^{m}\right|
\le \prod_{p|M}\sum_{m\ge 0}(p^{-1-\epsilon})^{m}
\ll_\epsilon 1\\
&\text{ and }&
\left|\prod_{p|M}(1-c_pp^{\epsilon+it})\right|
=
\prod_{p|M}\left|1-c_pp^{\epsilon+it}\right|
\le\prod_{p|M}\left(1+p^\epsilon\right) 
\le \prod_{p|M} p^{1+\epsilon}
\ll_\epsilon M^{1+ 2\epsilon}.
\end{eqnarray*}
Putting $s=1+\epsilon+it$ in \eqref{eq6} and using the above 
estimates along with the first inequality, we get the second 
inequality.
\end{proof}

The next proposition provides convexity bound for 
Rankin-Selberg $L$-function $R(f,g ; s)$. 

\begin{prop}\label{lem20}
For square-free integers $N_1,N_2$, let $f\in S_{k_1}^{new}(N_1),
~g\in S_{k_2}^{new}(N_2)$ be normalized Hecke eigenforms with 
$f\neq g$ and $N:= \text{lcm}[N_1,N_2]$. Then for any 
$t\in \R,~\epsilon>0$ and $1/2< \sigma <1$, one has 
\begin{equation}
R(f,g; \sigma+it)\ll_{\epsilon} N^{2(1-\sigma+\epsilon)}
\left(1+\frac{k_2-k_1}{2}\right)^{1-\sigma+\epsilon}
\left(\frac{k_1+k_2}{2}\right)^{1-\sigma+\epsilon}
(3+|t|)^{2(1-\sigma)+\epsilon}.
\end{equation}
\end{prop}

To prove this proposition, we shall use the following strong 
convexity principle due to Rademacher. 

\begin{prop}[Rademacher \cite{R}]\label{prop21}
Let $g(s)$ be continuous on the closed strip $a\le \sigma \le b$, 
holomorphic and of finite order on $a <  \sigma < b$. Further 
suppose that 
$$
|g(a+it)|\le E|P+a+it|^\alpha,
\tab 
|g(b+it)|\le F|P+b+it|^\beta
$$
where $E,F$ are positive constants and $P,\alpha,\beta$ are 
real constants that satisfy 
$$
P+a>0, \tab \alpha\ge \beta.
$$
Then for all $a<\sigma<b$ and for all $t\in \R$, we have 
$$
|g(\sigma+it)|\le (E|P+\sigma+it|^\alpha)^{\frac{b-\sigma}{b-a}}
(F|P+\sigma+it|^\beta)^{\frac{\sigma-a}{b-a}}.
$$
\end{prop}

\smallskip
We are now ready to prove Proposition \ref{lem20}.
\begin{proof}
We apply Proposition \ref{prop21} with 
\begin{eqnarray*}
&a&
=-\epsilon, \tab b=P=1+\epsilon, \tab F=C_2,\\
&E&
=C_1N^{2+4\epsilon}
\left(1+\frac{k_2-k_1}{2}\right)^{1+2\epsilon}
\left(\frac{k_1+k_2}{2}\right)^{1+2\epsilon}, 
~\alpha=2+ 4\epsilon, ~\beta=0,
\end{eqnarray*}
where $C_1,C_2$ are absolute constants 
depending only on $\epsilon$. Thus for any 
$-\epsilon < \sigma < 1+\epsilon$, we have 
$$
\zeta_N(2\sigma+ 2it)R(f,g; \sigma+it)\ll_{\epsilon} 
\left[N^{\frac{2+ 4\epsilon}{1+2\epsilon}} \left(1+\frac{k_2-k_1}{2}\right)
\left(\frac{k_1+k_2}{2}\right)\right]^{1-\sigma+\epsilon}
(1+\sigma+\epsilon+|t|)^{2(1-\sigma+\epsilon)}.
$$
Note that for $1/2<\sigma<1+\epsilon$, one knows
$$
|\zeta_N(2\sigma+2it)|^{-1}
\ll_\epsilon 
\log\log (N+2)\cdot |1+it|^\epsilon.
$$
Combining all together, we get \propref{lem20}.
\end{proof}

As an immediate corollary, we have 

\begin{cor}\label{cor11}
For square-free integers $N_1,N_2$, let $f\in S_{k_1}^{new}(N_1),
~g\in S_{k_2}^{new}(N_2)$ be normalized Hecke eigenforms with 
$f\neq g$ and $N:= \text{lcm}[N_1, N_2]$. Then for any $t\in \R$ and any 
$\epsilon>0$, one has 
$$
R(f,g; 3/4+it)
\ll_{\epsilon} 
\left[N^2 \left(1 + \frac{k_2-k_1}{2}\right)
\left(\frac{k_1+k_2}{2}\right) \right]^{1/4+\epsilon}
(3+|t|)^{1/2+\epsilon}.
$$
\end{cor}

\begin{prop}
For square-free integers $N_1,N_2$, let $f\in S_{k_1}^{new}(N_1),
~g\in S_{k_2}^{new}(N_2)$ be normalized Hecke eigenforms with 
$f \neq g$ and $N:= \text{lcm}[N_1, N_2]$. Then for 
any $\epsilon>0$, one has 
\begin{equation}\label{eq8}
\sum_{n \le x,  \atop{(n,N)=1 \atop {n \text{  square-free}} }}
\lambda_f(n)\lambda_g(n)\log^2(x/n)
\ll_\epsilon 
\left[ N^2 \left(1+\frac{k_2-k_1}{2}\right)
\left( \frac{k_1+k_2}{2} \right) \right]^{1/4+\epsilon}
x^{3/4}.
\end{equation}
\end{prop}

\begin{proof}
For any $\epsilon>0$, we know by Deligne's bound that
$$
\lambda_f(n)\lambda_g(n)
\ll_{\epsilon} 
n^{\epsilon}.
$$
Hence by Perron's summation formula (see page 56 and 
page 67 of \cite{RM}), we have 
$$
\sum_{n\le x \atop{(n,N)=1 \atop {n \text{  square-free}} }}
\lambda_f(n)\lambda_g(n)\log^2(x/n)
=
\frac{1}{\pi i}\int_{1+\epsilon-i\infty}^{{1+\epsilon+i\infty}}
R^{\flat}(f,g; s)\frac{x^s}{s^3}~ds
$$
where 
$$
R^{\flat}(f,g; s)= \prod_{p \nmid N} 
\left( 1 + \frac{\lambda_f(p)\lambda_g(p)}{ p^{s}}\right),
\phantom{m} \Re(s) > 1.
$$ 
Further 
\begin{equation}\label{sq}
R(f,g; s) = R^{\flat}(f,g; s) H(s),
\end{equation}
where $H(s)$ has an Euler product 
which converges normally for $\Re(s) > 1/2$.
Now we shift the line of integration to $\Re(s)=3/4$.
Observing that there are no singularities
in the vertical strip bounded by the lines  
with $\Re(s) = 1+\epsilon$ and $\Re(s)=3/4$
and using Proposition \ref{lem20} along
with \eqref{sq}, we have 
$$
\sum_{n\le x, \atop {(n,N)=1\atop {n \text{  square-free}} }}
\lambda_f(n)\lambda_g(n)\log^2(x/n)
=
\frac{1}{\pi i}\int_{3/4-i\infty}^{{3/4+i\infty}}
R^{\flat}(f,g; s)\frac{x^s}{s^3}~ds. 
$$
The above observations combined with Corollary \ref{cor11} 
then implies that
$$
\sum_{n\le x, \atop {(n,N)=1\atop {n \text{  square-free}}}}
\lambda_f(n)\lambda_g(n)\log^2(x/n)
\ll_\epsilon N^{1/2+\epsilon}
\left(1+\frac{k_2-k_1}{2}\right)^{1/4+\epsilon}
\left(\frac{k_1+k_2}{2}\right)^{1/4+\epsilon}
x^{3/4}.
$$
This completes the proof of the proposition.
\end{proof}

Our next lemma will play a key role in proving \thmref{th1}.

\begin{lem}\label{lem13}
For square-free integers $N_1,N_2$, let $f\in S_{k_1}^{new}(N_1),
~g\in S_{k_2}^{new}(N_2)$ be normalized Hecke eigenforms with 
$f\neq g$ and $N:= \text{lcm}[N_1, N_2]$.
Also assume that for any $\alpha \le 2$,
$\lambda_f(p^{\alpha})\lambda_g(p^{\alpha}) \ge 0$
for all $p^{\alpha} \le x$. Then for
$x \ge exp{(c \log^2 (\sqrt{\mathfrak{q}(f)} + \sqrt{\mathfrak{q}(g)} ))}$,
we have 
$$
\sum_{n \le x, \atop{(n,N)=1 \atop{ n \text{  square-free}} }} 
\lambda_f(n)\lambda_g(n) \gg \frac{x}{\log^2 x}.
$$
Here $\mathfrak{q}(f),~ \mathfrak{q}(g)$ are analytic conductors of 
Rankin-Selberg L-functions of $f$ and $g$ respectively with
\begin{equation}\label{cond}
\mathfrak{q}(f)\le  k_1^2 N_1^2\log\log N_1
\phantom{m}\text{     and   } \phantom{m}
\mathfrak{q}(g) \le k_2^2 N_2^2\log\log N_2
\end{equation}
and $c >0$ is an absolute constant.
\end{lem}

\begin{proof}
Using Hecke relation \eqref{eq2}, for any prime $(p,N)=1$,
we know that 
$$
\lambda_f(p^2) \lambda_g(p^2) = 
[\lambda_f(p)\lambda_g(p)]^2 - \lambda_f(p)^2 
- \lambda_g(p)^2 + 1. 
$$
By hypothesis $\lambda_f(p^2)\lambda_g(p^2) \ge 0$ for all  
$p \le \sqrt{x}$. Hence for any $p\le \sqrt{x}$ and $(p,N)=1$, we have
$$
\lambda_f(p)^2\lambda_g(p)^2 \ge  \lambda_f(p)^2 
+ \lambda_g(p)^2 - 1. 
$$
This implies that 
$$
\sum_{p \le \sqrt{x}, \atop (p,N)=1} \lambda_f(p)^2\lambda_g(p)^2 
\ge \sum_{p \le \sqrt{x}, \atop (p,N)=1}\lambda_f(p)^2 
+ \sum_{p \le \sqrt{x}, \atop (p,N)=1}\lambda_g(p)^2 
- \sum_{p \le \sqrt{x}, \atop (p,N)=1} 1.
$$
Using standard analytic techniques and prime number theorem for
Rankin-Selberg $L$-functions of $f$ and $g$ respectively
(see \cite{IK}, pages 94-95, 110-111 for further details), we see that 
$$
\sum_{p \le \sqrt{x}, \atop (p,N)=1} \lambda_f(p)^2\lambda_g(p)^2 
\ge c_1\frac{\sqrt{x}}{\log x}
$$
provided $x \ge exp{(c \log^2 (\sqrt{\mathfrak{q}(f)} + \sqrt{\mathfrak{q}(g)} ))}$,
where $c, c_1 > 0$ are absolute constants and $\mathfrak{q}(f), \mathfrak{q}(g)$
are as in equation \eqref{cond}.
Using the hypothesis 
$$
\lambda_f(p) \lambda_g(p)\ge 0
\phantom{m}\text{  and  }\phantom{m} 
\lambda_f(p^2) \lambda_g(p^2) \ge 0
$$ 
for all $p, p^2 \le x$ and assuming that 
$x \ge exp{(c \log^2 (\sqrt{\mathfrak{q}(f)} + \sqrt{\mathfrak{q}(g)} ))}$, 
we have
\begin{eqnarray*}
\sum_{n \le x, \atop{(n, N)=1 \atop{n \text{  square-free}}} } 
\lambda_f(n)\lambda_g(n) 
& \ge & 
\frac{1}{2}\sum_{\substack{p,q \le \sqrt{x}, \atop (pq, N) = 1, \\ p \ne q}} 
\lambda_f(pq)\lambda_g(pq) \\
& = & 
\frac{1}{2} (\sum_{p \le \sqrt{x}, \atop (p,N) = 1} 
\lambda_f(p)\lambda_g(p))^2 
-\frac{1}{2} \sum_{p \le \sqrt{x}, \atop (p, N) = 1} 
\lambda_f(p)^2\lambda_g(p)^2. 
\end{eqnarray*}
Now using Deligne's bound, we get
\begin{eqnarray*}
\sum_{n \le x, \atop{(n, N)=1 \atop{n \text{  square-free}}} } 
\lambda_f(n)\lambda_g(n) 
& \ge &  
\frac{1}{2}\left(\sum_{p \le \sqrt{x}, \atop (p, N) = 1} 
\lambda_f(p)\lambda_g(p)
\frac{\lambda_f(p)\lambda_g(p)}{4}~\right)^2
~-~ 8 \sum_{p \le \sqrt{x}, \atop (p, N) = 1} 1 \\
& = & \frac{1}{32} 
\left(\sum_{p \le \sqrt{x}, \atop (p, N) = 1} 
\lambda_f(p)^2\lambda_g(p)^2\right)^2
~+~ O \left( \frac{\sqrt{x}}{\log x} \right) \\
& \gg & \frac{x}{\log^2 x}.
\end{eqnarray*}
This completes the proof of the lemma.
\end{proof}

We are now in a position to complete the proof of \thmref{th1}.
\begin{proof} 
Assume that $\lambda_f(p^{\alpha})\lambda_g(p^{\alpha})\ge 0$ for 
all $p^{\alpha}\le x$ with $\alpha \le 2$. 
By Lemma \ref{lem13}, we see that  
\begin{equation}\label{eq9}
\sum_{n\le x/2, \atop {(n,N)=1, \atop{n \text{  square-free}} }}
\lambda_f(n)\lambda_g(n)\log^2(x/n)
~\gg
\sum_{n\le x/2, \atop{ (n,N)=1, \atop{n \text{   square-free}} }}
\lambda_f(n)\lambda_g(n)
~\gg~ 
\frac{x}{\log^2 x}
\end{equation}
provided 
$x \ge exp{(c \log^2 (\sqrt{\mathfrak{q}(f)} + \sqrt{\mathfrak{q}(g)} ))}$,
where $c > 0, \mathfrak{q}(f), \mathfrak{q}(g)$ are as in 
Lemma \ref{lem13}. Now comparing \eqref{eq8} and \eqref{eq9}, 
for any $\epsilon > 0 $, we have 
$$
x
~\ll_\epsilon~
\text{max}
\left\{exp{(c \log^2 (\sqrt{\mathfrak{q}(f)} + \sqrt{\mathfrak{q}(g)} ))}, 
\left[ N^2
\left(1+\frac{k_2-k_1}{2}\right)
\left(\frac{k_1+k_2}{2}\right) \right]^{1+\epsilon} \right\},
$$ 
where $c, \mathfrak{q}(f), \mathfrak{q}(g)$ are as before.
Here we have used Lemma 4 of Choie and Kohnen \cite{CK}. 
This completes the proof of \thmref{th1}.
\end{proof}

\section{Proof of the theorem \ref{th_2}}

We now state a Lemma which we shall use to
prove \thmref{th_2}.

\begin{lem}\label{signchangelemma}
 Let $\{a_n\}_{n \in \N}$ and $\{b_m\}_{m \in \N}$
 be two sequences of real numbers such that 
 \begin{enumerate}
\item 
$a_n = O(n^{\alpha_1}), \phantom{m} b_m = O(m^{\alpha_2})$,
\item
$\sum_{n, m \le x}a_nb_m \ll x^{\beta}$,
\item
$\sum_{n,m \le x}a_n^2b_m^2 = cx+O(x^{\gamma})$,
\end{enumerate}
where $\alpha_1,\alpha_2,\beta, \gamma \ge 0$ 
and $c>0$ such that $ \max\{\alpha_1+\alpha_2+\beta, \gamma\}< 1$.  
Then for any $r$ satisfying 
$$
\max\{\alpha_1 + \alpha_2+\beta, \gamma\}<  r < 1,
$$ 
there exists a sign change 
among the elements of the sequence $\{a_nb_m\}_{n,m \in \N}$
for $n,m \in [x,x+x^{r}]$.  Consequently, for sufficiently large $x$,
the number of sign changes among the elements
of the sequence $\{a_nb_m\}_{n,m \in \N}$ with $n,m \le x$
are $\gg x^{1-r}$. 
\end{lem}
 
\begin{proof}
Suppose that for any $r$ satisfying 
$$
\max\{\alpha_1+\alpha_2+\beta, \gamma\}< r< 1,
$$ 
the elements of the sequence $\{ a_nb_m \}_{n,m \in \N}$ 
have same signs in $[x,x+x^{r}]$. This implies that
$$
x^{r}
~\ll~
\sum_{x\leq n,m\leq x+x^{r}}a_n^2b_m^2
~\ll~ 
x^{\alpha_1+\alpha_2}
\sum_{x\leq n,m\leq x+x^{r}}a_nb_m
~\ll~ 
x^{\alpha_1+\alpha_2+\beta},
$$
which is a contradiction. This completes the proof of the Lemma.
\end{proof}
 
Lemma \ref{signchangelemma} can be thought of as a generalization
of a Lemma of Meher and Ram Murty (see \cite[Theorem 1.1]{MM1})
when $b_1 =1$ and $b_m=0$ for all $m > 1$. 
We are now in a position to prove \thmref{th_2}.

\begin{proof}
In order to apply Lemma \ref{signchangelemma}, we need to verify the 
following conditions for the elements of the sequence 
$\{\lambda_f(n)\lambda_g(n^2)\}_{n\in\N}$.
Note that 
\begin{enumerate}
\item 
Ramanujan-Deligne bound implies that
$$
\lambda_f(n)\lambda_g(n^2)=O_\epsilon(n^{\epsilon})
$$
for all $n\in\N$. 
 
\item 
By a recent work of L\"u \cite[Theorem 1.2(2)]{L} (see also Kumari 
and Ram Murty \cite{KM}),  one has
$$
\sum_{n\leq x}\lambda_f(n)\lambda_g(n^2)
\ll x^{5/7}(\log x)^{-\theta/2},
$$ 
where $\theta=1-\frac{8}{3\pi}=0.1512\ldots$

\item 
In the same paper, L\"u (see \cite[Lemma 2.3(ii)]{L} as well as 
Kumari and Ram Murty \cite{KM}) also proved that
$$
\sum_{n\leq x}\lambda_f(n)^{2}\lambda_g(n^2)^{2}
= 
cx+O(x^{\frac{17}{18}+\epsilon}),
$$
where $c>0$. 
\end{enumerate}
Theorem \ref{th_2} now follows from 
Lemma \ref{signchangelemma} by choosing
$a_n =  \lambda_f(n)$ and $b_m := \lambda_g(m^2)$
for all $m, n \in \N$ and considering the sequence
$\{ a_n b_n\}_{ n\in \N}$. 
\end{proof}

\section{Proof of the theorem \ref{th_3}}

Using equation \eqref{eq3}, one can write 
$$
\lambda_f(p)=2\cos \alpha_p
~~~ \text{  and  }~~~
\lambda_g(p)=2\cos \beta_p
$$
with $0\le \alpha_p, \beta_p\le \pi$. Using the Hecke 
relation \eqref{eq2} for any prime $(p,N_1N_2)=1$, 
one has 
\begin{equation}\label{eq4}
\lambda_f(p^m)=
\begin{cases}
(-1)^m (m+1)  &\text{ if } \alpha_p=\pi;\\
m+1 & \text{ if } \alpha_p=0;\\
\frac{\sin (m+1)\alpha_p}{\sin \alpha_p} 
& \text{ if } 0< \alpha_p<\pi.
\end{cases}
\end{equation}
and 
\begin{equation}\label{eq5}
\lambda_g(p^m)=
\begin{cases}
(-1)^m (m+1)  &\text{ if } \beta_p=\pi;\\
m+1 & \text{ if } \beta_p=0;\\
\frac{\sin (m+1)\beta_p}{\sin \beta_p} 
& \text{ if } 0< \beta_p<\pi.
\end{cases}
\end{equation}

\thmref{th_3} now follows from the following four
cases. 

\noindent 
{\bf Case (1):} 
When $\alpha_p=0 \text{ or }\pi$ and $\beta_p=0 \text{ or }\pi$, 
then by the equation \eqref{eq4} and equation \eqref{eq5}, we see that  
$$
\{m\in \N~|~ a_{f}(p^m)a_{g}(p^m)\neq0\}=\N.
$$ 
In this case all elements of the sequence 
$\{a_f(p^m)a_g(p^m)\}_{ m \in \N}$ are non-zero.

\noindent 
{\bf Case (2):} 
Suppose that at least one of $\alpha_p, \beta_p$, say $\alpha_p=0 \text{ or }\pi$
and $\beta_p \in (0, \pi)$. If $\beta_p/\pi \not\in \Q$, there is nothing
to prove. Now if $\beta_p/\pi = \frac{r}{s}$ with $(r,s)=1$, then
we have
$$
\#\{m\leq x~|~ a_f(p^m)a_g(p^m) \ne 0\}
=\#\{m\leq x ~|~ a_g(p^m) \ne 0\}
= [x]- \left[\frac{x}{s} \right].
$$
Hence the set $\{m ~|~ a_f(p^m)a_g(p^m) \ne 0\}$ has postive density.

\noindent 
{\bf Case (3):} 
Suppose that $\alpha_p=\beta_p\in (0,\pi)$, i.e.
$\alpha_p/ \pi =\beta_p/\pi \in (0,1)$. 
If $\alpha_p/\pi \notin\Q$, then  
$a_{f}(p^m)a_{g}(p^m)\ne 0$ for all $m\in \N$ 
as $\sin m\alpha_p\ne 0$ for all $m\in \N$ .
If $\alpha_p / \pi \in\Q$, 
say $\alpha_p / \pi=\frac{r}{s}$, where $r,s\in \N$ with 
$(r,s)=1$, then we have $\sin m\alpha_p=0$ if and only 
if $m$ is an integer multiple of $s$ and hence  
\begin{equation*}
\#\{m\leq x: a_{f}(p^m)a_{g}(p^m)\neq0\}
=
\left[x \right]-\left[\frac{x}{s}\right].
\end{equation*}
Hence the set in \eqref{densityofset} has positive density.

\noindent 
{\bf Case (4):} 
Assume that $\alpha_p, \beta_p \in (0, \pi)$ with 
$\alpha_p\neq \beta_p$.
If both $\alpha_p / \pi, \beta_p / \pi~\notin~\Q$, then 
there is nothing to prove. 
Next suppose that one of them, say $\alpha_p / \pi \in\Q$ with 
$\alpha_p / \pi=\frac{r}{s}$ with $(r,s)=1$ and 
$\beta_p / \pi~\notin~\Q$. 
Then we have
$$
\#\{m\leq x: a_f(p^m)a_g(p^m) \ne 0\}
=\#\{m\leq x: a_f(p^m) \ne 0\}
= [x]- \left[\frac{x}{s} \right].
$$
Hence the set in \eqref{densityofset} has positive
density. 

Now let both $\alpha_p / \pi, \beta_p / \pi \in \Q$. If  
$\alpha_p / \pi = \frac{r_1}{s_1}$ and 
$\beta_p/\pi =\frac{r_2}{s_2}$ with 
$(r_i,s_i)=1$, for $1\leq i\leq 2$, then 
$$
\#\{m\leq x: a_f(p^m)a_g(p^m) \ne 0\}
=\# \left[\{m\leq x: a_f(p^m)\ne 0\} \cap \{m\leq x: a_g(p^m)\ne 0\}\right].
$$
Note that both $s_1$ and $s_2$ can not be $2$ as otherwise 
$\alpha_p= \beta_p$.
Since
\begin{eqnarray*}
\#\{m\leq x: a_f(p^m)a_g(p^m) = 0\} 
&=&
\#\left[\{m\leq x: a_f(p^m) = 0\} \cup \{m\leq x: a_g(p^m) = 0\} \right]\\
&\le& 
\left[\frac{x}{s_1} \right] + \left[\frac{x}{s_2}\right],
\end{eqnarray*}
the set in \eqref{densityofset} has positive density. 
This completes the proof of \thmref{th_3}.

\medskip

\section{Proof of Theorem \ref{th_4}}

Using Lemma \ref{lem6}, we see that for any $x \ge 2$ and 
$0<\delta< 1/2$
$$
\#\{p\le x : a_f(p^{m}) = 0 \text{ for some } m \ge 1\} 
\ll_{f,\delta} 
\frac{x}{(\log x)^{1+\delta}}~,
$$
where the implied constant depends only on $f$ and $\delta$. 
We have the same estimate for the form $g$ as well. 
Therefore for any $x \ge 2$ and $0<\delta< 1/2$, we have
\begin{equation}\label{preq}
\#\{p\leq x : a_f(p^{m})a_g(p^{m'})=0 \text{ for some } m, m'\ge 1\}
\ll_{f,g,\delta} 
\frac{x}{(\log x)^{1+\delta}},
\end{equation}
where the implied constant depends on $f, g$ and $\delta$. Hence 
\begin{eqnarray*}
&& \#\{p\leq x : a_f(p^{m})a_g(p^{m'})\ne 0 \text{ for all } m, m' \ge 1\} \\
&& = 
\pi(x) - \#\{p\leq x : a_f(p^{m})a_g(p^{m'})=0 \text{ for some } m, m' \ge 1\},
\end{eqnarray*}
where $\pi(x)$ denotes the number of primes up to $x$.
Now using prime number theorem as well as the identity 
\eqref{preq}, we have  
$$
  \#\{p\leq x : a_f(p^{m})a_g(p^{m'})\ne 0 \text{ for all } m, m' \ge 1\} 
  ~\sim~ \frac{x}{\log x}.
$$ 
Hence the set 
$$
\{p\in\mathcal{P} : a_f(p^{m})a_g(p^{m'})\ne 0 
\text{ for any integers } m, m' \ge 1 \}
$$
has natural density $1$.

\medskip

\section{Proof of theorem \ref{th_5}}

We keep the notations in this section as in section 5.
To prove \thmref{th_5}, we start by proving the
following Proposition.

\begin{prop}\label{th3}
Let 
$$
f(z) =\sum_{n=1}^\infty a_f(n)q^n \in S_{k_1}(N_1)
\phantom{m}\text{ and }\phantom{m}
g(z) =\sum_{n=1}^\infty a_g(n)q^n \in S_{k_2}(N_2)
$$ 
be two distinct normalized Hecke eigenforms. 
Then for any prime $p$ with $(p, N_1N_2)=1$, 
there exists an integer $m$ with $1 \le  m \le 4$ 
such that $a_{f}(p^m)a_{g}(p^m) \ne 0$. 
\end{prop}

\begin{proof}
Note that  $a_f(p^m)a_g(p^m) \ne 0$ 
is equivalent to $\sin(m+1)\alpha_p\sin(m+1)\beta_p
\ne 0$. If $a_f(p)a_g(p)\neq 0$, then we are done. 
Now suppose $a_f(p)a_g(p)= 0$, then 
either $a_f(p)=0$ or $a_g(p)=0$.

\noindent
{\bf Case (1):} 
If $a_f(p)=0=a_g(p)$, then $\alpha_p=\beta_p=\pi/2$. 
Hence we have 
$$
a_f(p^2)a_g(p^2)
=
p^{k_1+k_2-2}\ne 0.
$$

\noindent
{\bf Case (2): }
Suppose that at least one of $a_f(p), a_g(p) \ne 0$.
Without loss of generality assume that
$a_f(p)=0$ and $a_g(p)\ne  0$, then 
$\alpha_p=\pi/2$ and $\beta_p\ne \pi/2.$ Now if 
$\beta_p=0$ or $\pi$, then $a_g(p^2)= 3p^{k_2-1}$. 
Hence we have 
$$
a_f(p^2)a_g(p^2)
= - 3 p^{k_1 + k_2- 2} \ne 0.
$$
If $\beta_p \notin \{0,\pi/2, \pi\}$, then this implies that
 $a_g(p^2)=p^{(k_2-1)}\frac{\sin 3\beta_p}{\sin \beta_p}$.
Now if $a_f(p^2)a_g(p^2)=0$, then 
$\beta_p\in \{\pi/3, 2\pi/3 \}$ as $0<\beta_p<\pi$.
Then we have
$$
\frac {a_f(p^4)a_g(p^4)}{p^{2(k_1+k_2-2)}}
= \frac{2}{\sqrt{3}} \sin \frac{5\pi}{2}\sin\frac{ 5\pi}{3}
\phantom{m}\text{or} \phantom{m}
\frac {a_f(p^4)a_g(p^4)}{p^{2(k_1+k_2-2)}}
= \frac{2}{\sqrt{3}}\sin \frac{5\pi}{2} \sin\frac{10\pi}{3}.
$$
Since neither $\sin (5\pi/2)\sin (5\pi/3)$  nor $\sin (5\pi/2)\sin(10\pi/3)$
is equal to zero, this completes the proof of \propref{th3}. 
\end{proof}

\begin{proof}
We now complete the proof of the first part of
\thmref{th_5} by showing the existence of a prime $p \le 2\log N$ with 
$(p,N)=1$ and then using \propref{th3}. We know by a 
theorem of Rosser and Schoenfeld 
(see \cite[p. 70]{RS}) that 
$$
\sum_{p\le x}\log p>0.73 x 
\phantom{m} \text{for all }x\ge 41. 
$$
Using this, one can easily check that 
$$
\sum_{p\le x}\log p> \frac{x}{2} 
\phantom{m} \text{for all }x\ge 5. 
$$
Now consider the following product 
$$
\prod_{p\le 2\log N}p
=\text{exp}\left(\sum_{p\le 2\log N}\log p\right)>N,
$$
which confirms the existence of such a prime.
Proof of the second part of \thmref{th_5}
follows immediately by applying \propref{th3}.
\end{proof}

\medskip

\section{$\B$-free numbers and simultaneous non-vanishing 
in short intervals}
In this section, we first list certain properties of 
$\B$-free numbers and their distribution in short intervals 
to derive simultaneous non-vanishing of Hecke eigenvalues. 
Erd\"os \cite{E} introduced the notion of $\B$-free
numbers and showed the existence of these
numbers in short intervals. 
\begin{defn}
Let us assume that 
$$
\B:=\{b_1, b_2, ...\}\subset \N
$$ 
be such that 
$$
(b_i,b_j)=1 \text{ for }i\neq j
\phantom{m}\text{ and }\phantom{m}
\sum_{i\ge 1}\frac{1}{b_i}<\infty.
$$
One says that a number $n\in\N$ is $\B$-free 
if it is not divisible by any element of the set 
$\B$.
\end{defn}

The distribution of $\B$-free numbers in short intervals 
has been studied by several mathematicians 
(see \cite{BG}, \cite{SW}, \cite{ES},
\cite{W2}, \cite{Z}). 
Balog and Ono \cite{BO}
were first to use $\B$-free numbers to 
study non-vanishing of Hecke eigenvalues. 

For a non-CM cusp form $f\in S_k(N)$
with Fourier coefficients $\{ a_f(n) \}_{n \in \N}$, 
Serre (see \cite[page 383]{S}) defined the function
$$
i_f(n):= \text{max }\{m\in\N ~|~ a_f(n+j)=0 \text{ for all }0<j\le m\}
$$
which is now known as gap function. 
Alkan and Zaharescu \cite{AZ} proved that 
$$
i_{\Delta}(n) \ll_{\Delta} n^{1/4 + \epsilon}
$$
for Ramanujan $\Delta$-function.
Kowalski, Robert and Wu \cite{KRW}, using
distribution of $\B$-free numbers in short intervals
showed that 
$$
i_f(n)\ll_f n^{7/17+\epsilon}
$$
where $f \in S_{k}^{new}(N)$ is any normalized Hecke eigenform.
Recently, Das and Ganguly \cite{DG} showed that
$$
i_f(n) \ll_f n^{1/4 + \epsilon}
$$
for any $f \in S_k(1)$. 

In this article, we will study simultaneous non-vanishing
of Hecke eigenvalues using $\B$-free numbers.
This question was first considered by 
Kumari and Ram Murty \cite{KM}. 
We now introduce the set of $\B$-free numbers 
as constructed by Kowalski, Robert and Wu \cite{KRW}. 
These numbers will play an important role
in our work. 

Let $\p$ be a subset of $\mathcal{P}$ such that 
\begin{equation}\label{eq12}
\#(\p\cap [1,x])
\ll 
\frac{x^\rho}{(\log x)^{\eta_\rho}}
\end{equation}
where $\rho\in[0,1]$ and $\eta_\rho$'s are real numbers 
with $\eta_1>1$. Let us define 
\begin{equation}\label{eq13}
\B_\p:=\p \cup \{p^2 ~|~ p\in\mathcal{P}-\p\}.
\end{equation}
Write $\B_\p=\{b_i ~|~i\in\N\}$. Note that $(b_i,b_j)=1$ 
for all $b_i, b_j\in \B_\p$ with $b_i\neq b_j$. To show 
$\sum_{i\in\N}\frac{1}{b_i}<\infty$, it is enough to show that 
$\sum_{p\in\p}\frac{1}{p}<\infty$. Applying equation \eqref{eq12}
and partial summation formula, one has
$$
\sum_{\substack{p\le x, \\  p\in\p}}\frac{1}{p}
=
\frac{1}{x}\sum_{\substack{p\le x, \\ p\in\p}}
1+\int_2^x \frac{1}{t^2}(\sum_{\substack{p\le t, \\ p\in\p}}1) dt
\ll_\p \frac{x^{\rho-1}}{(\log x)^{\eta_\rho}}
+\int_2^x\frac{t^{\rho-2}}{(\log t)^{\eta_\rho}}dt
\ll_\p1.
$$
With these notations, 
Kowalski, Robert and Wu (see Corollary 10 of \cite{KRW})
proved the following Theorem.

\begin{thm}[Kowalski, Robert and Wu]\label{th15}
For any $\epsilon>0, x\ge x_0(\p,\epsilon)$ and 
$y\ge x^{\theta(\rho)+\epsilon}$, we have 
$$
\#\{x<n\le x+y ~|~ n\text{ is }\B_\p\text{-free}\}
\gg_{\p,\epsilon} y,
$$
where
\begin{equation}\label{eq14}
\theta(\rho):=
\begin{cases}
\frac{1}{4} & \text{ if }0\le \rho\le \frac{1}{3};\\
\frac{10\rho}{19\rho+7} & \text{ if }\frac{1}{3}< \rho\le \frac{9}{17};\\
\frac{3\rho}{4\rho+3} & \text{ if }\frac{9}{17}< \rho\le \frac{15}{28};\\
\frac{5}{16} & \text{ if }\frac{15}{28}< \rho\le \frac{5}{8};\\
\frac{22\rho}{24\rho+29} & \text{ if }\frac{5}{8}< \rho\le \frac{9}{10};\\
\frac{7\rho}{9\rho+8} & \text{ if }\frac{9}{10}< \rho\le 1.
\end{cases}
\end{equation}
\end{thm}

We now study simultaneous non-vanishing in 
short arithmetic progression using the distribution of $\B$-free numbers. 
The question of the distribution of 
$\B$-free numbers in short arithmetic progression was first considered by 
Alkan and Zaharescu \cite{AZ1}. In this direction, Wu and Zhai (see 
Proposition 4.1 of \cite{WZ}) have the following result
about distribution of $\B$-free numbers in short arithmetic progression. 

\begin{thm}[Wu and Zhai]\label{th15A}
Let $\B_\p$ be as in \eqref{eq13}. For any $\epsilon>0, x\ge x_0(\p,\epsilon)$,  
$y\ge x^{\psi(\rho)+\epsilon}$ and $1\le a\le q\le x^\epsilon$ with $(a,q)=1$, 
one has
$$
\#\{x<n\le x+y : n\text{ is }\B_\p\text{-free} \text{ and } n\equiv a (\text{mod } q)\}
\gg_{\p,\epsilon} y/q,
$$
where
\begin{equation}\label{eq_14}
\psi(\rho):=
\begin{cases}
\frac{29\rho}{46\rho+19} & \text{ if }~\frac{190}{323}< \rho\le \frac{166}{173};\\
\frac{17\rho}{26\rho+12} & \text{ if }~\frac{166}{173}< \rho\le 1.
\end{cases}
\end{equation}
\end{thm}

Using above results, we now have the following
non-vanishing Theorem for certain multiplicative function. 
\begin{thm}\label{th16}
Let $f : \N \to \C$ be a multiplicative function and let 
$N\ge 1$ be a positive integer. Define  
\begin{equation}\label{eq15}
\p_{f,N}:=\{p\in \mathcal{P} ~|~ f(p)=0\}\cup \{p\in\mathcal{P} ~|~ p|N\}.
\end{equation}
Also assume that $\p_{f,N}$ satisfies condition \eqref{eq12}. 
Then 
\begin{enumerate}
\item
For any $\epsilon>0, x\ge x_0(\p_{f, N},~\epsilon)$ and 
$y\ge x^{\theta(\rho)+\epsilon}$, we have 
$$
\#\{x<n\le x+y ~|~ (n,N)=1,~n \text{ square-free and } f(n)\neq 0\}
\gg_{\p_{f, N},~\epsilon} y,
$$
where $\theta(\rho)$ is as in \eqref{eq14}.

\item
For any $\epsilon>0, x\ge x_0(\p_{f, N}, ~\epsilon)$,  
$y\ge x^{\psi(\rho)+\epsilon}$ and $1\le a\le q\le x^\epsilon$ 
with $(a,q)=1$, we have 
$$
\#\{x<n\le x+y : (n,N)=1,~n \text{ square-free, } n\equiv a (\text{mod } q) 
\text{ and } f(n)\neq 0\}
\gg_{\p_{f, N},~\epsilon} y/q,
$$
where $\psi(\rho)$ is as in \eqref{eq_14}.
\end{enumerate}
\end{thm}

\begin{proof}
Define 
$$
\B_{\p_{f,N}} := \p_{f, N} \cup \{p^2 ~|~ p\in\mathcal{P}-\p_{f,N}\}.
$$ 
Then first part of \thmref{th16} now follows from \thmref{th15}.
Applying \thmref{th15A}, we get the second part of \thmref{th16}.
\end{proof}

As an immediate corollary, we have 
\begin{cor}\label{ellipticurve}
Let $E_1/\Bbb Q$ and $E_2/\Bbb Q$ be two 
non-CM elliptic curves which have the 
same conductor $N$. Let 
$$
L(E_{i}, s)=\sum_{n=1}^{\infty}a_{E_{i}}(n)n^{-s},\ \ i=1,2
$$
be their Hasse-Weil $L$-functions. If 
$f_{E_{i}}(z)=\sum_{n=1}^{\infty}a_{E_{i}}(n)q^{n}$ for 
$i=1,2$ are the associated weight two newforms, then 
\begin{enumerate}
\item
for any $\epsilon>0$ and $y\geq x^{33/94+\epsilon}$,
we have 
$$
\#\{x<n<x+y ~|~ n \text{ is square-free and }a_{E_{1}}(n)a_{E_{2}}(n)\neq 0\}
\gg_{E_{1},E_{2},\epsilon} y.
$$
\item 
For any $\epsilon>0, x\ge x_0(E_1, E_2, \epsilon)$,  
$y\ge x^{87/214+\epsilon}$ and $1\le a\le q\le x^\epsilon$ 
with $(a,q)=1$, we have 
\begin{eqnarray*}
\#\{x<n\le x+y &|& (n,N)=1,~n \text{ is square-free and }, n\equiv a (\text{mod } q) 
\text{ and } a_{E_1}(n)a_{E_2}(n)\neq 0\} \\
&& \gg_{E_1, E_2, \epsilon} y/q.
\end{eqnarray*}
\end{enumerate}
\end{cor}

\begin{proof}
Let $\pi_{E}(x)$ be the number of supersingular 
primes up to $x$ for a non-CM elliptic curve $E/\Q$.
By the work of Elkies \cite{Elkies} , we have
$$
\#\{p\leq x : a_{E}(p)=0\}
\ll_{E} x^{3/4}.
$$ 
Considering $f(n):=a_{E_1}(n)a_{E_2}(n)$, one easily sees 
that $\p_{f,N}$ satisfies condition \eqref{eq12} with $\rho=3/4$
and $\eta_{\rho} = 0$. 
Now by using \thmref{th16}, we get the Corollary.
\end{proof}

Kumari and Ram Murty  have proved similar results for
non-CM cusp forms which are newforms and normalized 
Hecke eigenforms of weight $k>2$. 

As a second corollary, we have the following simultaneous non-vanishing 
result for coefficients of symmetric power $L$-functions in short intervals. 

To state the corollary,  we need to introduce few more notations. 
Let $f\in S_k^{new}(N)$ be a normalized Hecke eigenform
with Fourier coefficients $\{ a_f(n) \}_{ n\in \N}$.
Set $\lambda_f(n)=a_f(n)/n^{(k-1)/2}$ and suppose that
for $p\nmid N$,  $\alpha_{f,p},\beta_{f,p}$ are the Satake 
$p$-parameter of $f$. Then 
the un-ramified $m$-th symmetric power $L$-function of $f$
is defined as follows: 
$$
L_{unr}(sym^m f,s)
:=\prod_{p\nmid N}\prod_{0\le j\le m}(1-\alpha_{f,p}^j\beta_{f,p}^{m-j}p^{-s})^{-1}
=:\sum_{n\ge 1}\lambda_f^{(m)}(n)n^{-s}.
$$
We now have the following corollary.
\begin{cor}\label{cor18}
Let $f \in S_{k_1}^{new}(N_1)$ and $g \in S_{k_2}^{new}(N_2)$ 
be normalized non-CM Hecke eigenforms. 
Let $N:= \text{lcm}[N_1,N_2]$. 
Then  
\begin{enumerate}
\item
for any $\epsilon>0$, $x\ge x_0(f,g,\epsilon)$ 
and $y\ge x^{7/17+\epsilon}$, we have 
$$
\#\{x<n\le x+y ~|~ n \text{ is square-free and }
\lambda_f^{(m)}(n)\lambda_g^{(m)}(n)\neq 0\}
\gg_{f,g,m,\epsilon}y.
$$

\item
For any $\epsilon>0, x\ge x_0(f, g,\epsilon)$,  
$y\ge x^{17/38+\epsilon}$ and $1\le a\le q\le x^\epsilon$ 
with $(a,q)=1$, we have 
$$
\#\{x<n\le x+y ~|~ (n,N)=1,~n \text{ square-free, } n\equiv a (\text{mod } q) 
\text{ and } \lambda_f^{(m)}(n)\lambda_g^{(m)}(n)\neq 0\}
\gg_{f, g,\epsilon} y/q.
$$
\end{enumerate}
\end{cor}

\begin{proof}
Let 
$$
\p_{f,g,m}:=\{p\in\mathcal{P} ~|~ p|N \text{ or } 
\lambda_f^{(m)}(p)\lambda_g^{(m)}(p)= 0\}
$$
Since $\lambda_f^{(m)}(p)=\lambda_f(p^m)$, using 
Lemma \ref{lem6}, we see that $\p_{f,g,m}$ satisfies condition 
\eqref{eq12}. Note that $f(n):= \lambda_f^{(m)}(n)\lambda_g^{(m)}(n)$ 
is a multiplicative function and hence we can apply \thmref{th16}, 
to complete the proof of Corollary \ref{cor18}.
\end{proof}

\begin{rmk}
Note that Corollary \ref{cor18} implies simultaneous 
non-vanishing of Hecke eigenvalues 
in sparse sequences. More precisely, 
let $f \in S_{k_1}^{new}(N_1)$ and $g \in S_{k_2}^{new}(N_2)$ 
be normalized non-CM Hecke eigenforms. 
Also let $N:= \text{lcm}[N_1,N_2]$. 
Then  
\begin{enumerate}
\item
for any $\epsilon>0$, 
$x\ge x_0(f,g,\epsilon)$ and $y\ge x^{7/17+\epsilon}$, 
we have 
$$
\#\{x<n\le x+y : (n,N)=1,~n \text{ is square-free and }
\lambda_f(n^m)\lambda_g(n^m)\neq 0\}
\gg_{f,g,m,\epsilon}y.
$$

\item
For any $\epsilon>0, x\ge x_0(f, g,\epsilon)$,  
$y\ge x^{17/38+\epsilon}$ and $1\le a\le q\le x^\epsilon$ 
with $(a,q)=1$, we have 
$$
\#\{x<n\le x+y : (n,N)=1,~n \text{ square-free, } n\equiv a (\text{mod } q) 
\text{ and } \lambda_f(n^m)\lambda_g(n^m)\neq 0\}
\gg_{f,g,\epsilon} y/q.
$$
\end{enumerate}
\end{rmk}

\bigskip
\noindent
{\bf Acknowledgments:} 
The authors would like to thank Jyoti Sengupta for
bringing to their notice the paper of Lau-Liu-Wu \cite{LLW}.
The authors would also like to thank Purusottam Rath for 
going through an earlier version of the paper.

\end{document}